\documentclass{article}%

\usepackage{amsmath,amssymb,amsfonts}%
\usepackage{theorem}%
\usepackage{color}%
\usepackage{hyperref}%
\usepackage{bbm}%

\theoremstyle{change}%
\newtheorem{definition}{Definition:}[section]%
\newtheorem{proposition}[definition]{Proposition:}%
\newtheorem{theorem}[definition]{Theorem:}%
{\theorembodyfont{\rmfamily}\newtheorem{remark}[definition]{Remark:}}%
{\theorembodyfont{\rmfamily}}%

\newenvironment{proof}
  {{\bf Proof:}}
  {\qquad \hspace*{\fill} $\Box$}%

\newcommand{\fg}{\mathfrak{g}}%
\newcommand{\fn}{\mathfrak{n}}%
\newcommand{\fh}{\mathfrak{h}}%

\newcommand{\Ad}{\operatorname{Ad}}%
\newcommand{\ad}{\operatorname{ad}}%

\newcommand{\inner}{\operatorname{int}}%
\newcommand{\rme}{\mathrm{e}}%

\newcommand{\CC}{\mathcal{C}}%
\newcommand{\UC}{\mathcal{U}}%
\newcommand{\DC}{\mathcal{D}}%
\newcommand{\XC}{\mathcal{X}}%

\newcommand{\R}{\mathbb{R}}%

\begin{document}

\title{Affine systems on Lie groups and outer invariance entropy}
\author{Adriano Da Silva\footnote{Supported by Fapesp grant $n^o$ 2013/19756-8}\\Instituto de Matem\'atica,\\
Universidade Estadual de Campinas\\
Cx. Postal 6065, 13.081-970 Campinas-SP, Brasil}
\date{\today}
\maketitle

{\bf Abstract. }Affine systems on Lie groups are a generalization of linear systems. For such systems, this paper studies what happens with the outer invariance entropy introduced by Colonius and Kawan \cite{FCCK}. It is shown that, as for linear case, the outer invariance entropy is given by the sum of the positive real parts of the eigenvalues of a derivation $\mathcal{D}^*$ that is associated to the drift of the system.


\section{Introduction}
 
The notion of outer invariance entropy was defined in \cite{FCCK} as the smallest rate of information about the state, above which a controller is able to prevent trajectories, starting in a compact subset $K$, from leaving a controlled invariant subset $Q$ of the state space. Such notion is closely related with the notion of feedback entropy defined by Nair, Evans, Mareels, and Moran \cite{NGMW}. General bounds for this new concept of entropy were established in Kawan \cite{CK2011}, \cite{CK1}, \cite{CK} but exactly formulas are quite hard to obtain. When the state space is a Lie group or a homogeneous space their intrinsic geometry helps in such task and some explicit formulas for the entropy where found in \cite{CKawan}, \cite{DaSilva} and \cite{SilKa}. In this paper we study the outer invariance entropy for affine systems on Lie groups that are a generalization of linear systems introduced in \cite{VASM} and \cite{VAJT}. As for the linear case one can associate to the drift of an affine system a derivation and show that the outer invariance entropy in this case is also given in terms of the positive real parts of the eigenvalues associated with such derivation.

The contents of this paper are the following: Section 2 introduces the definition of control systems and the outer invariance entropy. In Section 3 we define affine and linear systems on Lie groups and homogeneous spaces and prove some results for such systems. In Section 4 our main result is proved. We first obtain an upper bound in terms of the positive real part of the eigenvalues of the associated derivation. In order to get a similar lower bound we consider an induced system on a particular homogeneous space where we have an invariant measure which allows us to get rid of the eigenvalues with negative real part and obtain a formula for the outer invariance entropy. 

\section{Preliminaries}

In this Section we will introduce the concept of control systems and its outer invariance entropy. For the theory of control systems we refer to \cite{FCWK} and \cite{Son} and for more on outer invariance entropy the reader can consult \cite{FCCK}, \cite{CK2011}, \cite{CK1}, \cite{CK} and \cite{CKawan}.

\subsection{Control Systems}
Let $M$ be a $d$ dimensional smooth manifold. By a control system in $M$ we understand a family 
\begin{equation}
\label{controlsystem}
\dot{x}(t)=f(x(t), u(t)), \;\;\;\;u\in\mathcal{U}
\end{equation}
of ordinary differential equations, with a right-hand side $f: M\times\mathbb{R}^m\rightarrow TM$
satisfying $f_u:=f(\cdot, u)\in\mathcal{X}(M)$ for all $u\in\mathbb{R}^m$. For simplicity, we assume
that $f$ is smooth. 

The family $\mathcal{U}$ of admissible control functions is given by
$$\mathcal{U}=\{u:\mathbb{R}\rightarrow\mathbb{R}^m; \;\;\;u\mbox{ measurable and }u(t)\in \Omega \;\mbox{ a.e.}\}$$
where $\Omega\subset\mathbb{R}^m$ is a compact convex set called the {\bf control range} of the system such that $0\in\inner\Omega$. Smoothness of $f$ in the first argument guarantees that for each control function $u\in \mathcal{U}$ and each initial value
$x\in M$ there exists a unique solution $\phi(t, x, u)$ satisfying $\phi(0, x, u)=x$,
defined on an open interval containing $t=0$. Note that in general $\phi(t, x, u)$
is only a solution in the sense of Carath\'eodory, i.e., a locally absolutely
continuous curve satisfying the corresponding differential equation almost
everywhere. We assume w.l.o.g. that all such solutions can be extended to the
whole real line since we only consider solutions which do not leave a small neighborhood of a compact set. 
Hence, we obtain a mapping 
$$\phi:\mathbb{R}\times M\times\mathcal{U}\rightarrow M, \;\;\;(t, x, u)\mapsto\phi(t, x, u),$$
satisfying the {\bf cocycle property}
$$\phi(t+s, x, u)=\phi(t, \phi(s, x, u), \Theta_su)$$
for all $t, s\in\mathbb{R}$, $x\in M$, $u\in\mathcal{U}$ where  for $t\in\mathbb{R}$  the map $\Theta_t$ is the {\bf shift flow} on $\mathcal{U}$ defined by 
$$(\Theta_tu)(s):=u(t+s).$$
Instead of $\phi(t, x, u)$ we will usually write $\phi_{t, u}(x)$. Note that smoothness of the right-hand side $f$ implies smoothness of $\phi_{t, u}$.

\subsection{Outer Invariance Entropy}
Consider the control system (\ref{controlsystem}), and let $\varrho$ be a fixed metric on $M$ compatible with the given topology. Let $K, Q\subset M$ be nonempty sets. We say that $(K, Q)$ is an {\bf admissible pair} of (\ref{controlsystem}) if
\begin{itemize}
\item[1.] The set $K$ is compact;
\item[2.] For each $x\in K$ there exists $u\in\mathcal{U}$ such that $\phi(\mathbb{R}^+, x, u)\subset Q$ (in particular, $K\subset Q$).
\end{itemize}

For given $\tau, \varepsilon> 0$ a set $S\subset \mathcal{U}$ of control functions is called $(\tau, \varepsilon, K, Q)$-spanning if for all $x\in K$ there exists $u\in S$ with $\phi(t, x, u)\in N_{\varepsilon}(Q)$ for all $t\in[0, \tau]$. The minimal cardinality of such a set is denoted by $r_{\mathrm{inv}}(\tau, \varepsilon, K, Q)$, and the {\bf outer invariance entropy} of $(K, Q)$ is defined as 
$$h_{\mathrm{inv, out}}(K, Q; \varrho):=\lim_{\varepsilon\searrow 0}\limsup_{\tau\rightarrow\infty}\frac{1}{\tau}\ln r_{\mathrm{inv}}(\tau, \varepsilon, K, Q)$$

By Proposition 2.5 of \cite{CK2011} the outer invariance entropy is independent of uniformly equivalent metrics on $Q$. In particular, if $Q$ is a compact set the outer invariance entropy is independent of the metric and we will denote it just by $h_{\mathrm{inv, out}}(K, Q)$ without further remarks. 

Let $\dot{x}=f_1(x, u)$ and $\dot{y}=f_2(y, u)$ be control systems on $M$ and $N$ with corresponding solutions $\phi_1(t, x, u)$ and $\phi_2(t, y, u)$ and same set of admissible functions $\UC$ and let $\pi: M\rightarrow N$ be a continuous map with the semiconjugation property
$$\pi(\phi_1(t, x, u))=\phi_2(t, \pi(x), u), \;\;\;\mbox{ for all }x\in M, u\in\mathcal{U}, t\geq 0.$$
Further assume that $(K, Q)$ is an admissilble set of the system $\dot{x}=f_1(x, u)$ on $M$ such that $Q$ is a compact set. Then $(\pi(K), \pi(Q))$ is an admissible set of the system $\dot{y}=f_2(y, u)$ on $N$ and 
\begin{equation}
\label{nondecrease}
h_{\mathrm{inv, out}}(\pi(K), \pi(Q))\leq h_{\mathrm{inv, out}}(K,Q).
\end{equation}
The proof of the above can be found in  \cite{CKawan}, Proposition 2.13.

\section{Linear and Affine Systems}

In this section we introduce linear and affine systems. It is well known that the flow associated with a linear system form a one parameter group of automorphisms of $G$ and that allows us to associate a derivation of the Lie algebra to it. We will use such fact and the fact that any affine vector field has unique decompositions given by the sum of a linear vector field and right/left invariant vector field to obtain the derivation which will estimate the entropy. 

\subsection{Linear and Affine Vector Fields on Lie Groups}

Let $G$ be a connected Lie group and $\fg$ its Lie algebra. Denote by $X(G)$ the set of $\CC^{\infty}$ vector fields on $G$. The {\bf normalizer} of $\fg$ is by definition the set
$$\eta:=\mathrm{norm}_{X(G)}(\fg):=\{F\in\ X(G); \,\mbox{ for all }Y\in\fg, \;\;[F, Y]\in\fg\}$$

\begin{definition}
A vector field $F$ on $G$ is said to be {\bf affine} if it belongs to $\eta$. Moreover, if $F(e)=0$ we say that $F$ is {\bf linear}, where $e\in G$ stands for the neutral element of $G$.
\end{definition}

By the above definition, an affine vector field is such that the restriction of $\ad(F)$ to $\fg$ is a derivation of $\fg$. It is not hard to show that $\eta$ is actually a Lie subalgebra of $X(G)$ and that the map $F\mapsto\ad(F)$ is a Lie algebra morphism from $\eta$ into the set of the derivations of $\fg$.

The following result, whose proof can be found in \cite{VAJT}, Theorem 2.2, shows that an affine vector field can be writen uniquely as sum of a linear vector field and a left invariant one.

\begin{theorem}
The kernel of the map $F\mapsto \ad(F)$ is the set of left invariant vector fields. An affine vector field $F$ can be uniquely decomposed as 
$$F=\XC+Z$$
where $\XC$ is linear and $Z$ is left invariant.
\end{theorem}

Let $X\in\fg$ be a right invariant vector field. If we denote by $i$ the inversion in $G$ we have that the vector field $i_*X$ is a left invariant vector field and it is equal to $-X$ at the neutral element $e\in G$. Therefore, the vector field $\XC=X+i_*X$ is a linear vector field. 

\bigskip
Under the assumption that $G$ is connected, it was proven in \cite{JPh} that we can also decompose an affine vector field $F\in\eta$ uniquely as $F=\XC^*+Y$ where $\XC^*$ is a linear vector field and $Y$ a right invariant vector field. The relation between both decompositions is given by 
$$Y=-i_*Z\;\;\mbox{ and }\;\;\XC^*=\XC+Z+i_*Z.$$
One should note also that, unless $\ad(i_*Z)\equiv 0$, we have
$$\ad(\XC^*)=\ad(\XC)+\ad(i_*Z)\neq \ad(\XC)=\ad(F)$$

Concerning linear vector fields, Theorem 1 of \cite{JPh} gives us the following equivalences:
\begin{itemize}
\item[(i)] $\XC$ is linear;
\item[(ii)] The flow of $\XC$ is a one parameter group of automorphisms;
\item[(iii)] $\XC$ verifies
$$\XC(gh)=(dL_g)_h\XC(h)+(dR_h)_g\XC(g), \;\;\mbox{ for any }\;\;g, h\in G.$$
\end{itemize}
Let us denote by $\DC$ the derivation of the Lie algebra $\fg$ given by $\mathcal{D}Y=-[\mathcal{X}, Y], \mbox{ for all }Y\in\mathfrak{g}$, that is, $\mathcal{D}=-\mathrm{ad}(\mathcal{X})$. For all $t\in\mathbb{R}$ and $Y\in\fg$ we have that
\begin{equation}
\label{flow}
(d\psi_t)_e=\mathrm{e}^{t\mathcal{D}} \mbox{ and that }\psi_t(\exp Y)=\exp(\mathrm{e}^{t\mathcal{D}}Y)
\end{equation}
where $\psi_t$ stands for the flow of $\XC$.

For an affine vector field $F$ we have that its associated flow is complete (see \cite{JPh}, Proposition 3) and, if $\alpha_t(e)$ stands for the solution of $F$ at $e\in G$, we have that
\begin{equation}
\label{expressionslinear}
\alpha_t(g)=R_{\alpha_t(e)}(\psi_t(g))=L_{\alpha_t(e)}(\psi^*_t(g))
\end{equation}
where $\psi_t$ and $\psi^*_t$ are the flows of $\XC$ and $\XC^*$, respectively. In fact, denote by $\zeta(t)$ the curve $R_{\alpha_t(e)}(\psi_t(g))$. Then
$$\frac{d}{dt}\zeta(t)=(dL_{\psi_t(g)})_{\alpha_t(e)}\frac{d}{dt}\alpha_t(e)+(dR_{\alpha_t(e)})_{\psi_t(g)}\frac{d}{dt}\psi_t(g)$$
$$=(dL_{\psi_t(g)})_{\alpha_t(e)}F(\alpha_t(e))+(dR_{\alpha_t(e)})_{\psi_t(g)}\XC(\psi_t(g))$$
$$=Z(\zeta(t))+(dL_{\psi_t(g)})_{\alpha_t(e)}\XC(\alpha_t(e))+(dR_{\alpha_t(e)})_{\psi_t(g)}\XC(\psi_t(g))$$
$$=Z(\zeta(t))+\XC(\zeta(t))=F(\zeta(t)),$$
where we used above the decomposition $F=\XC+Z$. By unicity of solutions we have that $\alpha_t(g)=R_{\alpha_t(e)}(\psi_t(g))$. By using the decomposition $F=\XC^*+Y$ one can show in the same way that $\alpha_t(g)=L_{\alpha_t(e)}(\psi^*_t(g))$. 

\begin{remark}
We should note that the above gives us in particular that
$$\psi_t=C_{\alpha_t(e)}\circ\psi^*_t.$$
\end{remark}

\subsection{Linear and Affine Systems on Lie Groups}

 An {\bf affine system} on a Lie group $G$ is a control system of the form
\begin{equation}
\label{affinesystem}
\dot{g}(t)=F(g(t))+\sum_{j=1}^mu_j(t)X_j(g(t)), \;\;\;\;u=(u_1, \ldots, u_m)\in\UC
\end{equation}
where the drift vector field $F$ is affine and $X_j$ are right invariant vector fields. We say that the above system is a {\bf linear system} if $F$ is a linear vector field. 

The usual example of an affine system is the one on $\mathbb{R}^n$ given by
$$\dot{x}(t)=Ax(t)+Bu(t)+Y; \;\;A\in\R^{n\times n}, B\in\R^{n\times m}, Y\in\R^n.$$
Since the right (left) invariant vector fields on the abelian Lie algebra $\mathbb{R}^n$ are given by constant vectors we have
$$\dot{x}(t)=F(x(t))+\sum_{i=1}^mu_i(t)b_i, \hspace{.5cm} B=(b_1| b_2| \cdots|b_m)$$
where $F(x):=Ax+Y$, that is, it is an affine system in the sense of (\ref{affinesystem}). 
 
\bigskip
For a given affine vector field $F\in\eta$ we can consider the associated system
\begin{equation}
\label{linearsystem}
\dot{g}(t)=\XC(g(t))+\sum_{i=1}^mu_i(t)X_i(g(t))
\end{equation}
where $F=\XC+Z$ with $\XC$ linear and $Z$ left invariant. We have then the following result concerning the solutions of an affine system.

\begin{proposition}
\label{rightinv}
For given $u\in\mathcal{U}$, $t\in\mathbb{R}$ and $g\in G$ the solutions of the affine system (\ref{affinesystem}) satisfies
$$\phi(t, g, u)=\phi_{t, u}\alpha_t(g)=L_{\phi_{t, u}}(\alpha_t(g))$$
where $\phi_{t, u}$ is the solution of the linear system (\ref{linearsystem}) associated to $F$ starting at the neutral element $e\in G$.

\end{proposition}

\begin{proof}
Let us consider the curve $\zeta(t)$ given by 
$$\zeta(t)=\phi_{t, u}\alpha_t(g).$$
We have that $\zeta(0)=g$ and 
$$\dot{\zeta}(t)=(dL_{\phi_{t, u}})_{\alpha_t(g)}\frac{d}{dt}\alpha_t(g)+(dR_{\alpha_t(g)})_{\phi_{t, u}}\frac{d}{dt}\phi_{t, u}=$$
$$=(dL_{\phi_{t, u}})_{\alpha_t(g)}F(\alpha_t(g))+(dR_{\alpha_t(g)})_{\phi_{t, u}}\biggl\{\XC(\phi_{t, u})+\sum_{j=1}^mu_j(t)X_j(\phi_{t, u})\biggr\}$$
$$=Z(\zeta(t))+\sum_{j=1}^mu_j(t)X_j(\zeta(t))$$
$$+\,\biggl\{(dL_{\phi_{t, u}})_{\alpha_t(g)}\XC(\alpha_t(g))+(dR_{\alpha_t(g)})_{\phi_{t, u}}\XC(\phi_{t, u})\biggr\}$$
and since $\XC$ is linear we have that
$$(dL_{\phi_{t, u}})_{\alpha_t(g)}\XC(\alpha_t(g))+(dR_{\alpha_t(g)})_{\phi_{t, u}}\XC(\phi_{t, u})=\XC(\phi_{t, u}\alpha_t(g))=\XC(\zeta(t))$$
which gives us  
$$\dot{\zeta}(t)=\XC(\zeta(t))+Z(\zeta(t))+\sum_{j=1}^mu_j(t)X_j(\zeta(t))=F(\zeta(t))+\sum_{j=1}^mu_j(t)X_j(\zeta(t))$$
and by unicity that $\zeta(t)=\phi(t, g, u)$ showing the desired.

\end{proof}


\subsection{Linear Systems on Homogeneous Spaces}

Let $H$ be a closed connected Lie subgroup of the Lie group $G$ and consider the homogeneous space $G/H$ and $\pi:G\rightarrow G/H$ the canonical projection.

For a given linear vector field $\mathcal{X}$ on $G$, we want to assure the existence of a vector field on $G/H$ that is $\pi$-related to $\mathcal{X}$. Such a vector field exists if, and only if,
$$\mbox{ for all }x\in G \mbox{ and } y\in H, \hspace{1cm} \pi(\psi_t(xy))=\pi(\psi_t(x)).$$
But $\pi(\psi_t(xy))=\psi_t(x)\psi_t(y)H$ and the preceding condition is equivalent to 
$$\mbox{ for all }y\in H, t\in\mathbb{R}\hspace{1cm} \psi_t(y)\in H.$$
Thus $\mathcal{X}$ is $\pi$-related to a vector field $f$ on $G/H$ if, and only if, $H$ is invariant under the flow of $\mathcal{X}$. Moreover, since $H$ is connected we have by (\ref{flow}) that $H$ is invariant by the flow of $\XC$ if, and only if, its Lie algebra $\fh$ is $\DC$-invariant.

For a given $\mathcal{D}$-invariant Lie subalgebra $\mathfrak{h}$ of $\mathfrak{g}$, let $H\subset G$ be the connect Lie subgroup with Lie algebra $\mathfrak{h}$ and assume that $H$ is closed. Since $\mathcal{X}$ and $f=(d\pi)_*\mathcal{X}$ are $\pi$-related, its respective solutions $\psi_t$ and $\Psi_t$ satisfies 
$$\Psi_t\circ\pi=\pi\circ\psi_t,\;\;\;t\in\mathbb{R}.$$

\begin{definition}
Let $G$ be a connected Lie group and $H\subset G$ a closed subgroup. A vector field $f$ on a homogeneous space $G/H$ is said to be {\bf affine} if there is $F\in\eta$ such that $F$ and $f$ are $\pi$-related, where $\pi:G\rightarrow G/H$ is the canonical projection.
\end{definition}

Using then the decomposition $F=\XC^*+Y$ with $\XC^*$ linear and $Y$ right invariant we have the following result (see Proposition 5 of \cite{JPh}).

\begin{proposition}
Let $H\subset G$ be a closed subgroup. Then $F$ and $f:=(d\pi)F$ are $\pi$-related if, and only if, $\XC^*$ and $f^*:=(d\pi)\XC^*$ are $\pi$-related, hence if, and only if, $\fh$ is invariant under $\DC^*$.
\end{proposition}

Let then $H$ to be a closed Lie subgroup that is invariant by the flow of $\XC^*$. The {\bf induced affine system} on $G/H$ is given by
\begin{equation}
\label{inducedaffine}
\dot{x}(t)=f(x(t))+\sum_{i=1}^mu_i(t)f_i(x(t))
\end{equation}
where $f:=(d\pi)F$, $f_i=(d\pi)_*X_i$ for $i=1, \ldots m$ and $u=(u_1, \ldots, u_m)\in\mathcal{U}$. 

Since the systems (\ref{affinesystem}) and (\ref{inducedaffine}), are $\pi$-related, we have that 
$$\pi(\phi(t, g, u))=\Phi(t, \pi(g), u), \;\;\; t\in\mathbb{R}, g\in G, u\in\mathcal{U}$$
where $\Phi$ is the solution of (\ref{inducedaffine}) above. Moreover, if gor any $g\in G$ we denote by $\mathcal{L}_g$ the translation on $G/H$, we have that
$$\Phi(t, x, u)=\pi(\phi_{t, u}\cdot\alpha_t(g))=\phi_{t, u}\cdot\pi(\alpha_t(g))=\mathcal{L}_{\phi_{t, u}}(\Lambda_t(x)),$$
for any $ t\in\mathbb{R}, \,x=gH\in G/H$ and $u\in\mathcal{U}$, where $\phi_{t, u}$ is as before the solution of the linear system (\ref{linearsystem}) at the neutral element $e$ of $G$ and $\Lambda_t(x)=\pi(\alpha_t(g))$ is the flow associated to the affine vector field $f$ on $G/H$.

\section{Upper and Lower Bounds}

It was shown in \cite{DaSilva} that the invariance entropy of any admissible pair $(K, Q)$ of a linear system, with $Q$ compact coincides with the sum of the real parts of the eigenvalues of the associated derivation. Our goal now is show that the same holds for the outer invariant entropy of the affine system (\ref{linearsystem}). 

\subsection{Upper Bound}

Before give an upper bound for the outer invariance entropy we will need the notion of topological entropy. 

Let $(X, d)$ be a metric space and $\varphi:\mathbb{R}\times X\rightarrow X$ be a flow over $X$. For a given compact set $K\subset X$ and $\varepsilon, \tau>0$ we say that a set $S_{\mathrm{top}}\subset K$ is $(\tau, \varepsilon)$-spanning set for $K$ if, for every $y\in K$ there exist $x\in S_{\mathrm{top}}$ such that 
$$\varrho(\varphi_t(x), \varphi_t(y))<\varepsilon\;\;\;\mbox{ for all } \;\;\;t\in [0, \tau].$$
If we denote by $r_{\tau}(\varepsilon, K)$ the minimal cardinality of a spanning set, the topolo\-gical entropy of $\varphi$ over $K$ is defined as
$$h_{\mathrm{top}}(\varphi, K; \varrho):=\lim_{\varepsilon\searrow 0}\limsup_{\tau\rightarrow\infty}\frac{1}{\tau}r_{\tau}(\varepsilon, K)$$
and the {\bf topological entropy} of $\varphi$ as
$$h_{\mathrm{top}}(\varphi; \varrho):=\sup_{K \mbox{\scriptsize{compact}}}h_{\mathrm{top}}(\varphi, K; \varrho).$$

As for the outer invariance entropy, the topological entropy $h_{\mathrm{top}}(\varphi; \varrho)$ is independent of uniformly equivalent metrics. 

The topological entropy of a flow $\varphi$ coincedes with the time one map $\varphi_1$, that is, $h_{\mathrm{top}}(\varphi; \varrho)=h_{\mathrm{top}}(\varphi_1; \varrho)$ (see for instance Lemma 2.1 in \cite{FCCK}). If $\varphi$ is a flow of automorphisms on a Lie group $G$, we have by Corollary 16 and Proposition 10 of \cite{RB}  that
$$h_{\mathrm{top}}(\varphi; \varrho_L)=h_{\mathrm{top}}(\varphi_1; \varrho_L)=\sum_{|\beta|>1}\log |\beta|$$
where $\beta$ are the eigenvalues of $(d\varphi_1)_e$ and $\varrho_L$ is a left invariant metric of $G$.

We have then the following Theorem.

\begin{theorem}
\label{upperouter}
Let $(K, Q)$ be an admissible pair of the affine system (\ref{affinesystem}) on $G$ and let $\varrho_L$ be a left invariant metric on $G$. Then, the outer invariant entropy satisfies
$$h_{\mathrm{inv, out}}(K, Q; \varrho_L)\leq \sum_{\lambda_{\DC^*}> \;0}\lambda_{\DC^*}$$
where $\lambda_{\DC^*}$ are the real parts of the eigenvalues of the derivation $\DC^*=-\ad(\XC^*)$ for the decomposition $F=\XC^*+Y$, with $\XC^*$ linear and $Y$ right invariant. 
\end{theorem}

\begin{proof}
By proposition \ref{rightinv} and the left invariance of the metric we have that two solutions $\phi_{t, u}(g)$ and $\phi_{t, u}(g')$, satisfies
$$\varrho_L(\phi_{t, u}(g'), \phi_{t, u}(g))=\varrho_L(\phi_{t, u}\cdot\alpha_t(g'), \phi_{t, u}\cdot\alpha_t(g))=\varrho_L(\alpha_t(g'), \alpha_t(g)).$$

We have by (\ref{expressionslinear}) that $\alpha_t(g)=L_{\alpha_t(e)}(\psi_t^*(g))$ which give us also by left invariance of the metric
$$\varrho_L(\alpha_t(g'), \alpha_t(g))=\varrho_L(\alpha_t(e)\cdot\psi^*_t(g'), \alpha_t(e)\cdot\psi^*_t(g))=\varrho_L(\psi^*_t(g'), \psi^*_t(g)).$$

Let then $S_{\mathrm{top}}\subset K$ be a minimal $(\tau, \varepsilon)$-spanning set for $K$ of the flow $\psi^*_t$, that is, for all $g'\in K$ there exists $g\in S_{\mathrm{top}}$ such that
$$\varrho_L(\psi^*_t(g), \psi^*_t(g'))<\varepsilon \,\,\,\,\,\mbox{ for all }t\in[0, \tau].$$
Since $(K, Q)$ is admissible and $S_{\mathrm{top}}\subset K$ there exists, for each $g\in S_{\mathrm{top}}$, $u_{g}\in\UC$ such that $\phi_{t, u_g}(g)\in Q$ for all $t\in\mathbb{R}$. Then for all $g'\in K$, there exists $u_g$ such that 
$$\varrho_L(\phi_{t, u_g}(g'), \phi_{t, u_g}(g))=\varrho_L(\phi_{t, u_g}\cdot\,\alpha_t(g'), \phi_{t, u_g}\cdot\,\alpha_t(g))$$
$$=\varrho_L(\alpha_t(g), \alpha_t(g))=\varrho_L(\psi^*_t(g), \psi^*_t(g))<\varepsilon$$
for all $t\in[0, \tau]$, that is, $\phi_{t, u_g}(g')\in N_{\varepsilon}(Q)$ showing that $\{u_g; \; g\in S_{\mathrm{top}}\}$ is a $(\tau, \varepsilon)$-spanning set for $(K, Q)$ and implying that 
$$h_{\mathrm{inv, out}}(K, Q; \varrho_L)\leq h_{\mathrm{top}}(K, \psi^*; \varrho_L) \leq h_{\mathrm{top}}(\psi^*; \varrho_L)=h_{\mathrm{top}}(\psi^*_1; \varrho_L).$$

Since $\psi^*_1$ is an automorphism we have
$$h_{\mathrm{top}}(\psi_1^*; \varrho_L)=\sum_{\beta; \,|\beta|>1}\log|\beta|$$
where $\beta$ are the eigenvalues of $(d\psi^*_1)_e$. Moreover, since 
$$(d\psi^*_1)_e=\mathrm{e}^{\DC^*}$$
we have that the eigenvalues of $(d\psi^*_1)_e$ are given by the exponential of the eigenvalues of $\DC^*$ and consequently $|\beta|=\mathrm{e}^{\lambda_{\DC^*}}$, where $\lambda_{\DC^*}$ is the real part of an eigenvalue of $\DC^*$. Then
$$h_{\mathrm{top}}(\psi^*_1; \varrho_L)=\sum_{\beta; \,|\beta|>1}\log|\beta|=\sum_{\lambda_{\DC^*}>0}\lambda_{\DC^*}$$
showing the Theorem.

\end{proof}

\subsection{Lower Bound}

In order to obtain the lower bound, we need to get rid of the eigenvalues of $\DC^*$ with negative real parts.

Consider then the generalized eigenspaces associated with the derivation $\DC^*$ given by 
$$\fg_{\beta}=\{X\in\mathfrak{g}; (\DC^*-\beta)^nX=0 \mbox{ for some }n\geq 1\}$$
where $\beta$ is an eigenvalue of $\DC^*$. By Proposition 3.1 of \cite{SM} we have for two eigenvalues $\gamma, \beta$ of $\DC^*$ that 
\begin{equation}
\label{eigenvalues}
[\fg_{\gamma}, \fg_{\beta}]\subset\fg_{\gamma+\beta}
\end{equation}
if $\gamma+\beta$ is an eigenvalue of $\DC^*$ and zero otherwise.

With the above, we have that
$$\fg=\fg^{+, \,0}\oplus\fn$$
where $\fg^{+, \,0}$ and $\fn$ are Lie subalgebras, with $\fn$ nilpotent, defined as 
$$\fg^{+, \,0}=\bigoplus_{\beta; \,\mathrm{Re}(\beta)\geq 0}\fg_{\beta}\hspace{1cm} \mbox{ and }\hspace{1cm}\fn=\bigoplus_{\beta; \,\mathrm{Re}(\beta)<0}\fg_{\beta}.$$
Moreover, Proposition 2.9 of \cite{Sil2} assures that the connected nilpotent subgroup $N$ with Lie algebra $\fn$ is closed in $G$.

\bigskip

A measure $\mu$ on a homogeneous space $G/H$ is said to be a {\bf $G$-invariant} Borel measure if for any $g\in G$ and any Borel set $A$ of  $G/H$ we have that $\mu(\mathcal{L}_g(A))=\mu(A)$. It is a well known fact (cf. \cite{Knapp} Theorem 8.36) that when such measure exists we have, for any continuous function $f$ with compact support, that
$$\int_Gf(g)d_G(g)=\int_{G/H}\int_Hf(gh)\,d_H(h)\,d\mu(\bar{g}),$$
where $d_G$, $d_H$ are the left invariant Haar measures on $G$ and $H$, respectively. Since for any compact set $K\subset G$, $\bar{\mathbbm{1}}_K(\bar{g})=\int_H\mathbbm{1}_K(gh)\,d_H(h)$ is a bounded positive function and $\bar{\mathbbm{1}}_K>0$ iff $\mathbbm{1}_{\pi(K)}>0$  we have that
$$\mu(\pi(K))>0 \;\;\;\; \mbox{ if }\;\;\;\;d_G(K)>0.$$

The {\bf modular function} $\Delta_G$ of a Lie group $G$ is defined as
$$\Delta_G(g)=|\mathrm{det} (\Ad(g))|.$$
By Theorem 8.36 of \cite{Knapp} a quotient space $G/H$ admits a unique (up to scalar) $G$-invariant Borel measure if, and only if, $(\Delta_G)_{|H}=\Delta_H$. 

By considering the Lie subgroup $N$ as above, we have that $\Delta_N=1$, since $N$ is nilpotent. Also by the nilpotency of $N$ we have that for any $g\in N$ there is $X\in\fn$ such that $g=\exp X$. Moreover, by equation (\ref{eigenvalues}) we have that $\ad(X)$ is a nilpotent map which give us that
$$\Delta_G(g)=\Delta_G(\exp X)=|\mathrm{det}\Ad(\exp X))|=|\mathrm{det} (\mathrm{e}^{\ad(X)})|=|\mathrm{e}^{\mathrm{tr}(\ad(X))}|=1.$$
and implies that $(\Delta_G)_{|N}=\Delta_N=1$ and consequently that $G/N$ admits a unique $G$-invariant Borel measure.

\bigskip
With the above we are in conditions to give the main result of this paper.

\begin{theorem}
\label{improved}
Let $(K, Q)$ be an admissible pair of the system (\ref{affinesystem}) such that $d_G(K)>0$ and $Q$ is a compact set. It holds that 
$$h_{\mathrm{inv, out}}(K, Q)=\sum_{\lambda_{\DC^*}>0}\lambda_{\DC^*}$$
where $\lambda_{\DC^*}$ are the real parts of the eigenvalues of $\DC^*$.
\end{theorem}

\begin{proof}
We just have to prove the inequality
$$h_{\mathrm{inv, out}}(K, Q)\geq\sum_{\lambda_{\DC^*}>0}\lambda_{\DC^*}.$$
Since the affine system (\ref{affinesystem}) on $G$ is $\pi$-related with the affine system (\ref{inducedaffine}) on $G/N$ and $Q$ is compact we have by equation (\ref{nondecrease}) that
$$h_{\mathrm{inv, out}}(K, Q)\geq h_{\mathrm{out, inv}}(\bar{K}, \bar{Q})$$
where $\bar{K}=\pi(K)$ and $\bar{Q}=\pi(Q)$. For any $\tau, \varepsilon>0$ let $S=\{u_1, \ldots, u_k\}\subset\UC$ be a minimal $(\tau, \varepsilon)$-spanning set of the admissible pair $(\bar{K}, \bar{Q})$ and consider the subsets 
$$K_j:=\{x\in \bar{K}; \Phi([0, \tau], x, u_j)\subset N_{\varepsilon}(\bar{Q})\}.$$
Using the contintuity of the map $\Phi_{t, u}$ is easy to see that $K_j$ and $\Phi_{\tau, u_j}(K_j)$ are Borel sets for every $j =1, \ldots, k$. Also $\bar{K}=\cup_jK_j$ and $\Phi_{\tau, u_j}(K_j)\subset N_{\varepsilon}(\bar{Q})$. Consider then the $G$-invariant measure $\mu$ on $G/N$. We have that
$$\mu(\Phi_{\tau, u_j}(K_j))=\mu(\mathcal{L}_{\phi_{\tau, u_j}}(\Lambda_{\tau}(K_j)))=\mu(\Lambda_{\tau}(K_j)).$$
Also, since $\alpha_t(g)=\alpha_t(e)\psi_t^*(g)$ we have that $\Lambda_t(x)=\mathcal{L}_{\alpha_t(e)}(\Psi_t^*(x))$ where $\Psi^*_t=\pi\circ\psi_t^*$ is the flow associated with the linear vector field $f^*=(d\pi)\XC^*$ on $G/N$. We have then that
$$\mu(\Lambda_{\tau}(K_j))=\mu(\mathcal{L}_{\alpha_{\tau}(e)}(\Psi_{\tau}^*(K_j)))=\mu(\Psi_{\tau}^*(K_j))=\int_{K_j} |\det(d\Psi^*_{\tau})_{\bar{g}}|d\mu(\bar{g}).$$

Since for any $g\in G$ the translation $\mathcal{L}_g$ on $G/N$ and $\Psi^*_{\tau}$ satisfies $\Psi^*_{\tau}\circ \mathcal{L}_g=\mathcal{L}_{\psi^*_{\tau}(g)}\circ \Psi^*_{\tau}$ we have that
$$(d\Psi^*_{\tau})_{\bar{g}}=(d\mathcal{L}_{\psi^*_{\tau}(g)})_{N}\circ (d\Psi^*_{\tau})_{N}\circ (d\mathcal{L}_{g^-1})_{\bar{g}}$$
an by the $G$-invariance of $\mu$ that $|\det(d \mathcal{L}_g)_{\bar{h}}|=1$ for any $\bar{h}\in G/N$. Consequently 
$$\det(d\Psi^*_{\tau})_{\bar{g}}=\det(d\Psi^*_{\tau})_{N}.$$
Also, Proposition 3.5 of \cite{DaSilva} gives us that 
$$(d\Psi^*_{\tau})_{N}=\rme^{\tau\DC^*_{|\fg^{+, 0}}}$$
and consequently that
$$\det(d\Psi^*_{\tau})_{\bar{g}}=\det(d\Psi^*_{\tau})_{N}=\det(\mathrm{e}^{\tau\DC^*|\fg^{+, 0}})=\mathrm{e}^{\tau\,\mathrm{tr}(\DC^*_{|\mathfrak{g}^{+, 0}})}=\mathrm{e}^{\tau\left(\sum_{\lambda_{\DC^*}>0}\lambda_{\DC^*}\right)}$$
which gives us that
$$\mu(\Phi_{\tau, u_j}(K_j))=\mathrm{e}^{\tau\left(\sum_{\lambda_{\DC^*}>0}\lambda_{\DC^*}\right)}\mu(K_j).$$

Let $j_0\in\{1, \ldots, k\}$ such that $\mu(K_{j_0})\geq\mu(K_j)$ for $j=1, \ldots, k$. Since $\Phi_{\tau, u_j}(K_j)\subset N_{\varepsilon}(\bar{Q})$ we have $\mu(\Phi_{\tau, u_j}(K_j))\leq \mu(N_{\varepsilon}(\bar{Q}))$ and consequently that
$$\mu(\bar{K})\leq \sum_{j=1}^k\mu(K_j)\leq k\cdot\mu(K_{j_0})=k\cdot\mu(\Phi_{n, u_{j_0}}(K_{j_0}))\,\mathrm{e}^{-\tau\left(\sum_{\lambda_{\DC^*}>0}\lambda_{\DC^*}\right)}$$ 
$$\hspace{4.1cm}\leq k\cdot\mu(N_{\varepsilon}(\bar{Q}))\mathrm{e}^{-\tau\left(\sum_{\lambda_{\DC^*}>0}\lambda_{\DC^*}\right)}$$
which implies that
$$r_{\mathrm{inv}}(\tau, \varepsilon, \bar{K}, \bar{Q})\geq \frac{\mu(\bar{K})}{\mu(N_{\varepsilon}(\bar{Q}))}\,\mathrm{e}^{\tau\left(\sum_{\lambda_{\DC^*}>0}\lambda_{\DC^*}\right)}.$$
Since we are assuming $d_G(K)>0$ we have that $\mu(\bar{K})=\mu(\pi(K))>0$. Also, since $Q$ is a compact set, $\bar{Q}=\pi(Q)$ is also a compact set and consequently, for $\varepsilon>0$ small enough, $\mu(N_{\varepsilon}(\bar{Q}))<\infty$. Dividing by $\tau$, taking the logarithm and the limsup we have then that 
$$\limsup_{\tau\rightarrow\infty}\frac{1}{\tau}\frac{\mu(\bar{K})}{\mu(N_{\varepsilon}(\bar{Q}))}=0$$
which gives us  
$$h_{\mathrm{inv, out}}(\bar{K}, \bar{Q})\geq \sum_{\lambda_{\DC^*}>0}\lambda_{\DC^*}$$
and concludes the proof.
\end{proof}

\bigskip
\begin{remark}
The above results generalize the results in \cite{DaSilva} for linear systems. In fact, if $F$ is linear vector field we have that both linear vector fields $\XC$ and $\XC^*$ coincides with $F$ and consequently $\DC=\DC^*$ which implies
$$h_{\mathrm{inv, out}}(K, Q)=\sum_{\lambda_{\DC}>0}\lambda_{\DC}=\sum_{\lambda_{\DC^*}>0}\lambda_{\DC^*}$$ 
for any admissible pair $(K, Q)$ of the linear system (\ref{linearsystem}).

A case where we have also the equality is in the Euclidean abelian case, because in this case $[,]\equiv 0$. But in general we have that $\DC^*=\DC+\ad(i_*Z)\neq \DC$ which show us that, different from the linear case, the outer invariance entropy of the system is not determined by the derivation $\ad(F)$ associated with the affine vector field $F$.
\end{remark}

\end{document}